\documentclass[a4paper,10pt]{article}

\usepackage[margin=4cm]{geometry}
\usepackage{subfigure}
\usepackage{color}
\usepackage{mathrsfs}
\usepackage{amsmath,amssymb,latexsym,amsthm}
\usepackage{enumitem}
\usepackage[T1]{fontenc}
\usepackage{lmodern}
\usepackage{microtype}

\usepackage[utf8]{inputenc}
\usepackage{mathtools} 
\usepackage{a4wide,amsmath,amssymb}
\usepackage[toc,page]{appendix}
\usepackage{bbm}
\usepackage{hyperref}
\usepackage{url}
\usepackage{courier}
\usepackage{easyReview}
\theoremstyle{plain}
\newtheorem{theorem}{\textbf{Theorem}}[section]
\newtheorem{proposition}[theorem]{\textbf{Proposition}}
\newtheorem{lemma}[theorem]{\textbf{Lemma}}

\theoremstyle{definition}
\newtheorem{definition}[theorem]{\textbf{Definition}}

\theoremstyle{remark}

\newtheorem{remark}[theorem]{Remark}

\numberwithin{equation}{section}

\def\N{\mathbb{N}}

\def\supp{\textrm{supp}}

\title{Kolmogorov $\varepsilon$-entropy of numerical solutions for scalar conservation laws with convex flux}

\RequirePackage[normalem]{ulem} 
\RequirePackage{color}\definecolor{RED}{rgb}{1,0,0}\definecolor{BLUE}{rgb}{0,0,1} 
\def\sideremark#1{\ifvmode\leavevmode\fi\vadjust{
		\vbox to0pt{\hbox to 0pt{\hskip\hsize\hskip1em
				\vbox{\hsize3cm\tiny\raggedright\pretolerance10000
					\noindent #1\hfill}\hss}\vbox to8pt{\vfil}\vss}}}

\makeatletter
\@namedef{subjclassname@2020}{\textup{2020} Mathematics Subject Classification}
\makeatother
\begin{document}
\author{Fabio Ancona\footnote{Dip. di Matematica “Tullio Levi Civita”, Univ. di Padova, {\textsc ancona@math.unipd.it}} \and Alessio Basti\footnote{Dip. di Ingegneria e Geologia, Univ. ``G. d'Annunzio'' di Chieti-Pescara,
		viale Pindaro 42, 65127 Pescara (Italy), { \textsc alessio.basti@unich.it, fabio.camilli@unich.it.}} \and
 Fabio Camilli$^\dag$}
\maketitle

\begin{abstract}
Building on the information-theoretic perspective of P.~D.~Lax [\textit{Proc.\ Sympos., Math.\ Res.\ Center, Univ.\ Wisconsin}, 1978], we establish a two-sided quantitative compactness estimate for numerical solutions of scalar conservation laws with a uniformly convex flux, expressed in terms of Kolmogorov $\varepsilon$-entropy. We prove that, under specific grid constraints, conservative, monotone finite-difference schemes satisfying a discrete one-sided Lipschitz condition (OSLC) preserve the $1/\varepsilon$ Kolmogorov entropy scaling of the corresponding exact entropy solution set, matching the bounds obtained by De~Lellis and Golse [\textit{Comm.\ Pure Appl.\ Math.}\ \textbf{58} (2005)] and by Ancona, Glass, and Nguyen [\textit{Comm.\ Pure Appl.\ Math.}\ \textbf{65} (2012)]. Specifically, the upper bound follows from the discrete OSLC, while the lower bound relies on a uniform approximation argument on a bounded-variation precursor class. Our results show that prototypical first-order methods are high-resolution in Lax's sense. Finally, we abstract the lower bound mechanism into a general transfer principle, discuss implications for information recovery via post-processing, and indicate directions for future work.
\end{abstract}

\noindent
{\footnotesize \textbf{AMS-Subject Classification:} 35L65, 65M06, 94A17}.

\section{Introduction}\label{intro}
\setcounter{equation}{0}

The aim of this work is to develop an information-theoretic framework for the analysis of numerical schemes for conservation laws, and more generally for partial differential equations (PDEs), inspired by the ideas of P.~D.~ Lax \cite{Lax78} (see also \cite{Lax02}). This perspective is intended to run in parallel with, and complement, the classical theory of error estimates for numerical schemes. In particular, in his seminal work, he proposed comparing the finite complexity of the numerical solution set with that of its exact solution counterpart. This approach reveals intrinsic constraints on numerical fidelity: a scheme whose order of complexity is lower than that of the exact solution family at scale $\varepsilon$ cannot, in principle, approximate every element of the family with an high fidelity. To quantify this notion of ``resolution'', Lax employed the Kolmogorov $\varepsilon$-entropy (henceforth simply the $\varepsilon$-entropy), defined as follows.

\begin{definition}[$\varepsilon$-Entropy]
Let $(X,d)$ be a metric space and let $K\subset X$ be totally bounded. For $\varepsilon>0$, denote by $N_\varepsilon(K)$ the minimal number of sets of diameter at most $2\varepsilon$ required to cover $K$. The $\varepsilon$-entropy of $K$ (relative to $X$) is then
\[
\mathcal H_{\varepsilon}(K\mid X)\;:=\;\log_2 N_\varepsilon(K).
\]
We call a cover of $K$ by subsets of $X$ having diameter no larger than $2\varepsilon$
an $\varepsilon$-cover of $K$.
\end{definition}

This notion of the complexity of a set quantifies the number of bits required to represent it at a given resolution scale \cite{KT59}. In \cite{Lax78}, depending on whether numerical schemes preserve the order of magnitude of the $\varepsilon$-entropy associated with the PDE problem, schemes are classified as \textit{high-resolution} or \textit{low-resolution}. We explain this more precisely below. 

Let $\mathcal{D}$ be a set of admissible initial data for the conservation law of interest, and let $S^T: \mathcal{D} \to X$ be the solution operator that maps an initial datum $d \in \mathcal{D}$ to the exact solution at time $T$ in a metric space $(X, \|\cdot\|)$. For a mesh with spacing $\Delta x$ denote by $P_{\Delta x}$ the projection operator that maps initial data to their discrete representations $\mathcal{D}_{\Delta x}$ on the mesh; write $u_0:=P_{\Delta x}(d)$. Let $S^N_{\Delta x}$ denote a numerical scheme that produces a discrete solution at discrete time level $N$ such that $T=N\Delta t$, and let $I$ be a linear interpolation operator such that $I\big(S^N_{\Delta x}(u_0)\big)\in X$.

Let $\delta(\mathcal D)$ be the supremum of the approximation error over the set of initial data:
\begin{equation}\label{eq:delta_def}
\delta(\mathcal D) := \sup_{d \in \mathcal{D}} \big\| S^T(d) - I(S^N_{\Delta x}(u_0)) \big\|.
\end{equation}

We will say that a method is \emph{accurate} if $\delta(\mathcal D)\to 0$ as $\Delta x\to 0$. This is therefore a notion tied to the possibility of reconstructing the exact solutions in the mesh-refinement limit. The information-theoretic notion, by contrast, pertains to the theoretical possibility of an optimal reconstruction of the solution at a fixed spatial resolution $\Delta x$. To explain this, let us define the sets of exact and reconstructed numerical solutions respectively as
\[
{S}^T (\mathcal D) := \{ S^T(d) : d \in \mathcal{D} \},
\]
and
\[
{I}\big({S}_{\Delta x}^{N}(\mathcal D_{\Delta x})\big)  := \{ I(S^N_{\Delta x}(u_0)) : u_0=P_{\Delta x}(d),\ d \in \mathcal{D} \}.
\]

A \textit{high-resolution} scheme is characterized by the ability to preserve the entropy scaling of the exact solution family, i.e. ${I}({S}_{\Delta x}^{N}(\mathcal D_{\Delta x}))$ scales in $\varepsilon$ as ${S}^T (\mathcal D)$. This means that the sets have the same complexity and contain the same amount of information, thereby allowing the error to be reduced (potentially even if the scheme itself is not accurate) by means of information recovery via appropriate post-processing (see Section~4 for more details). Conversely, a \textit{low-resolution} method is one for which the entropy scaling of ${I}({S}_{\Delta x}^{N}(\mathcal D_{\Delta x}))$ is strictly coarser than that of ${S}^T (\mathcal D)$; such methods suffer from unavoidable approximation errors $\delta(\mathcal D)$.

To quantify a lower bound for this error, we now restate, in a slightly different form (see the Appendix for the proof), a fundamental theorem that connects the entropy of the exact solution set and that of the numerical solution set via an inequality. Note that the theorem utilizes $\delta(\mathcal D)$ exactly as defined in \eqref{eq:delta_def}:

\begin{theorem}[Lax's Inequality for $\varepsilon$-Entropy]\label{lax_thm}
Under the definitions above, the following inequality holds between the $\varepsilon$-entropy of the exact solution set and that of the numerical reconstruction set:
\begin{equation}\label{eq:lax_theorem}
\mathcal{H}_{2\delta(\mathcal D)}\big({S}^T(\mathcal D)\mid X\big)
\le 
\mathcal{H}_{\delta(\mathcal D)}\big({I}({S}_{\Delta x}^{N}(\mathcal D_{\Delta x})) \mid X\big).
\end{equation}
\end{theorem}

Therefore, for low-resolution methods, establishing a lower bound for the left-hand side and an upper bound for the right-hand side yields a lower bound on the error $\delta(\mathcal D)$. Indeed, assume, e.g., that for sufficiently small $\varepsilon$
\[
\mathcal H_{\varepsilon}(S^T(\mathcal D)\mid X)
\geq C_1 \varepsilon^{-\alpha},
\qquad
\mathcal H_{\varepsilon}(I(S^N_{\Delta x}(\mathcal D_{\Delta x}))\mid X)
\leq C_2 \,\varepsilon^{-\beta},
\]
with $\alpha>\beta>0$. Then \eqref{eq:lax_theorem} implies
\begin{equation}
\delta(\mathcal D)
\ge
\left(\frac{C_1}{2^\alpha C_2}\right)^{1/(\alpha-\beta)}.
\end{equation}
In this sense, whenever the numerical entropy grows strictly more slowly
than the exact one, the loss of metric complexity translates into an
intrinsic obstruction to uniform approximation.
Therefor, it is  fundamental to obtain estimates for both the $\varepsilon$-entropy of the exact solution set and that of the numerical reconstruction set. In the past two decades a substantial body of work has produced quantitative estimates of the $\varepsilon$-entropy for $\mathcal{S}^T (\mathcal D)$ in different cases. For instance, in the case of a scalar conservation law with convex flux in one space dimension, De~Lellis and Golse \cite{DLG}, building on Lax's compactness result \cite{Lax54} for the entropy semigroup in $L^1_{\mathrm{loc}}$, obtained upper bounds on the $\varepsilon$-entropy in $L^1_{\mathrm{loc}}$ of the set of entropy solutions at time $t>0$ generated by bounded sets of initial data in $L^1$. Complementing these upper bounds, Ancona, Glass and Nguyen \cite{AON1} derive matching lower bounds of order $1/\varepsilon$ and then extend the quantitative compactness estimates to other cases \cite{AON2,AON3}. These estimates therefore provide the natural benchmark against which numerical approximations should be evaluated.

To the best of our knowledge, no result in the literature currently provides estimates for $\varepsilon$-entropy of the numerical reconstruction set in any setting. Thus, no direct quantitative comparison with the $\varepsilon$-entropy of an exact solution family has yet been obtained. The present paper takes a first step toward filling this gap, while implicitly addressing the following concrete questions:
\begin{enumerate}[label=\textbf{Q\arabic*}]
    \item \label{Q1} In the case of conservation laws with convex flux, under what hypotheses does a numerical solution set inherit the same $\varepsilon$-entropy scaling as the exact entropy solution set?    
    \item \label{Q2} What is the relationship between the approximation error \eqref{eq:delta_def} and the $\varepsilon$-entropy resolution and does the former play a role in characterizing the latter?
\item \label{Q3}  What implications does this notion of resolution have for the design, analysis, and error control of numerical methods?
\end{enumerate}

\section{Main results}\label{main}
\setcounter{equation}{0}

We consider the scalar conservation law in one space dimension
\begin{equation}\label{eq:cl}
u_t + f(u)_x = 0,\qquad t\ge 0,\; x\in\mathbb{R},
\end{equation}
under the standing assumptions that the flux $f:\mathbb{R}\to\mathbb{R}$ is 
twice continuously differentiable
and uniformly strictly convex,
\begin{equation}\label{eq:flux-assump}
f''(u)\ge c>0\quad\text{for all }u\in\mathbb{R},
\end{equation}
and that the normalization
\begin{equation}\label{eq:zero-speed}
f'(0)=0
\end{equation}
holds. For every $d\in L^1(\mathbb{R})\cap L^\infty(\mathbb{R})$ we denote by $S^T(d)$ the Kruzhkov entropy admissible weak solution of \eqref{eq:cl} at time $T\ge0$ (see, e.g., \cite{Dafermos:Book}). 

It is classical that $S^T$ is compact (for fixed $T>0$) from $L^1(\mathbb{R})$ into $L^1_{\mathrm{loc}}(\mathbb{R})$. Quantitative two-sided bounds for the $\varepsilon$-entropy of $S^T$ of order $1/\varepsilon$, for compact classes of initial data \begin{equation}\label{eq:continuous-class}
\mathcal D_{[L,m,M]}:=\Big\{d\in L^1(\mathbb{R})\cap L^\infty(\mathbb{R}):\ \supp(d)\subset[-L,L],\ \|d\|_{L^1}\le m,\ \|d\|_{L^\infty}\le M\Big\},
\end{equation}
have been obtained in earlier works; see in particular \cite{DLG,AON1}. 

The purpose of the present work is to provide quantitative bounds for the numerical solution family produced by finite-difference schemes. We fix a uniform spatial mesh size $\Delta x>0$ and a time step $\Delta t>0$ with $\lambda=\Delta t/\Delta x$. We set $x_j=j\Delta x$ and $t^n=n\Delta t$. We then consider a finite-difference scheme on the infinite grid based on a numerical flux $g$:
\begin{equation}\label{eq:fd}
u_j^{n+1}=u_j^n-\lambda\bigl[g(u_{j+1}^n,u_j^n)-g(u_j^n,u_{j-1}^n)\bigr],\qquad n \geq 0, \quad j\in\mathbb{Z}.
\end{equation}

We define the corresponding class of admissible discrete initial data via the standard cell-averaging projection operator $P_{\Delta x}$:
\begin{equation}\label{eq:discrete-class}
\mathcal D_{\Delta x} := \left\{ u_0 = P_{\Delta x}(d) : \ \left(P_{\Delta x}(d)\right)_j = \frac{1}{\Delta x} \int_{(j-1/2)\Delta x}^{(j+1/2)\Delta x} d(x) \, dx, \ d \in \mathcal D_{[L,m,M]} \right\}.
\end{equation}
We denote by $S_{\Delta x}^n$ the discrete evolution operator that maps $u^0$ to $u^n$. 
Since the continuous functions in $\mathcal D_{[L,m,M]}$ are supported in $[-L,L]$, their discrete projections $u_0$ have compact support; specifically, $(u_0)_j = 0$ for $|x_j| > L$. By the finite speed of propagation of the explicit scheme \eqref{eq:fd}, it is inherently guaranteed that $u^n = S_{\Delta x}^n(u_0)$ remains compactly supported for any finite time step $n$. In particular, the support of $u^n$ is contained in $[-L^n,L^n]$ with $L^n=L+n\Delta x$.

For our analysis, we fix a target index $N\in\mathbb{N}$ and set the final evaluation time $T := t^N = N\Delta t$, alongside the final support bound $L^N = L+N\Delta x$.

The following structural assumptions on the numerical scheme will be in force throughout the paper:
\begin{enumerate}[label=(\textbf{H\arabic*})]
    \item \textit{Consistency, conservativity, and monotonicity.}\label{hyp:general_cond}
      The finite-difference scheme \eqref{eq:fd} is conservative and consistent. The numerical flux
    $g:\mathbb{R}^2\to\mathbb{R}$ is locally Lipschitz continuous and monotone, in the sense that
    \[
        \partial_1 g \le 0, \qquad \partial_2 g \ge 0
    \]
    at every point where these derivatives exist, and the
Courant--Friedrichs--Lewy (CFL) condition holds:
\begin{equation}
    \lambda \sup_{\|u\|_{L^\infty}\le M}|f'(u)|\le 1,
    \qquad
    \lambda=\frac{\Delta t}{\Delta x}.
\end{equation}

    \item \textit{Discrete one-sided Lipschitz condition (OSLC).}\label{hyp:oslc}
    The forward explicit scheme satisfies a discrete OSLC estimate \cite{BO}: there exists $\beta>0$ such that,
    for
    \[
        \ell^+(u^n):= \max_j \frac{(u_{j+1}^n-u_j^n)_+}{\Delta x},
    \]
    one has
    \[
        \ell^+(u^n)\le \frac{1}{\beta\, t^n}
        \qquad\text{for every } n \text{ such that } t^n>0.
    \]
\end{enumerate}

Under these assumptions, classical results for monotone approximation schemes for scalar
conservation laws (see, e.g., \cite{kuz, cm}) imply that, for every initial datum
$d\in BV(\mathbb{R})$ with bounded support and $TV(d)\le C_{d}$, the reconstructed
numerical solution converges in $L^1(\mathbb{R})$ to the entropy solution at time $T=t^N$:
\begin{equation}\label{eq:convergence}
    \lim_{\Delta x\to 0}
    \bigl\| I(S_{\Delta x}^N(u_0)) - S^T(d)\bigr\|_{L^1(\mathbb{R})}=0.
\end{equation}
Moreover, the corresponding quantitative error bound depends only on $\Delta x$, $T$,
$C_{d}$, and the flux $f$.

For initial data in $BV(\mathbb{R})$, Kuznetsov's estimate \cite{kuz,hr} provides a worst-case fractional error bound of the form $\mathcal{O}(\Delta x^{1/2})$. Explicit knowledge of this convergence bound is mathematically pivotal to our lower estimate analysis, as it strictly defines the minimum scale $\varepsilon$ at which the numerical solution set can successfully emulate the information-theoretic complexity of the exact solution set.

Our main result establishes two-sided $\varepsilon$-entropy bounds of order $1/\varepsilon$ for the solution set $I(S_{\Delta x}^{N}(\mathcal D_{\Delta x}))$ at time $T$, matching up to multiplicative constants. Furthermore, these estimates mirror the theoretical bounds established for the underlying PDE problem. Our estimates rely on two distinct mechanisms. 

The upper bound is established relying on the intrinsic regularizing properties of the numerical operator. Specifically, we exploit the discrete OSLC and the stability of the scheme to map the interpolated numerical solutions into a class of bounded, non-decreasing functions, whose metric entropy can be explicitly estimated. 

Conversely, the strategy for the lower bound is to demonstrate that the numerical solution set contains an approximation of a specific target set included in the solution space of the PDE problem. If the forward numerical error on exact precursors of this target set is uniformly bounded, the discrete solutions must inherit its optimal $\varepsilon$-entropy scaling. We construct this lower bound target by taking advantage of the approach of \cite{AON1}. We first fix an amplitude parameter $h$ related to the PDE, satisfying
\begin{equation}\label{eq:h_param}
0 < h \le \min\left\{M, \frac{m}{2L}, \frac{L}{8T|f''(0)|}\right\},
\end{equation}
and define the corresponding shrunken support length 
\begin{equation}
\label{eq:suppLT}
L_T := L - 2T|f''(0)|\,h.
\end{equation}
Let \begin{equation}
\begin{split}
\label{eq:Adef}
    \mathcal{A}_{[L_T, Lh, h, (2T|f''(0)|)^{-1}]} := \bigg\{ v \in BV(\mathbb{R})\;\bigg|\; & \supp(v) \subset [-L_T, L_T], \ \|v\|_{L^1(\mathbb{R})} \le Lh,  \|v\|_{L^\infty(\mathbb{R})} \le h, \\
    & \ Dv \le (2T|f''(0)|)^{-1} \text{ in the sense of measures} \bigg\}
\end{split}
\end{equation}
be the admissible target space. As shown in \cite{AON1}, this specific set exhibits a $1/\varepsilon$ entropy scaling. For any target profile $v$ in this set, we define an \textit{exact precursor} as an initial datum $d \in \mathcal D_{[L,m,M]}$ whose evolution reaches $v$ at the target time $T$, namely $S^T(d) = v$. As we will establish in Lemma \ref{lem:uniform_tv_precursors}, all such precursors share a uniform total variation bound $C_{BV}$
dependning only on $L,M,T$ and $f''(0)$. This structural property is the crucial link: it guarantees that the discrete operator $S^N_{\Delta x}$ approximates the exact flow uniformly over this entire subset of precursors. We denote this bounded-variation subset of initial data as 
\begin{equation}
\label{eq:DLmMC-def}
\mathcal{D}_{[L,m,M,C_{BV}]}:=
\Big\{d \in \mathcal D_{[L,m,M]}
 \,:\, 
  TV(d) \le C_{BV} \Big\},
\end{equation}
and define the maximum forward numerical error over this class as:
\begin{equation}
\delta(\mathcal{D}_{[L,m,M,C_{BV}]}) := 
\sup\Big\{ \big\| S^T(d) - I(S^N_{\Delta x}(u_0)) \big\|_{L^1(\mathbb{R})} \,:\, d \in \mathcal{D}_{[L,m,M,C_{BV}]} \Big\}.
\end{equation}
This quantity represents the approximation error introduced in \eqref{eq:delta_def}, but specifically evaluated over the subset $\mathcal{D}_{[L,m,M,C_{BV}]}$ of initial data. By the assumptions on the scheme, and Kuznetsov's estimate \cite{kuz,hr}, this supremum is strictly finite, and it vanishes as $\Delta x \to 0$. 

We can now state the two-sided entropy estimate as follows:
\begin{theorem}[$\varepsilon$-entropy bounds for numerical solutions]
\label{thm:main}
Let $f:\mathbb{R}\to\mathbb{R}$ be a twice continuously differentiable map satisfying
\eqref{eq:flux-assump} and \eqref{eq:zero-speed}.
Assume that the finite-difference scheme \eqref{eq:fd}, with spatial mesh size $\Delta x>0$
and time step $\Delta t$, fulfills the structural conditions
\ref{hyp:general_cond}--\ref{hyp:oslc}.
Given $L,m,M>0$, $N\in\mathbb{N}$, and $\alpha\in(0,1]$, define
\[
L^N := L + N\Delta x,
\qquad
t^N := N\Delta t.
\]
Then the set of piecewise-linear reconstructed numerical solutions satisfies:
\begin{enumerate}
\item \textbf{Upper bound.}
For every $\varepsilon$ such that
\[
0 < \varepsilon \le
\frac{2L^N+\Delta x}{6}
\left(
\frac{2L^N+\Delta x}{\beta t^N}
+ \sqrt{\frac{2m}{\beta t^N}}
\right),
\]
one has
\begin{equation}\label{u_bound}
\mathcal{H}_{\varepsilon}\!\left(
I\big(S^N_{\Delta x}(\mathcal D_{\Delta x})\big)
\,\middle|\, L^1(\mathbb{R})
\right)
\le
\frac{\Gamma^+}{\varepsilon},
\end{equation}
where
\[
\Gamma^+ :=
\frac{4(2L^N + \Delta x)^2}{\beta t^N}
+
4(2L^N + \Delta x)\sqrt{\frac{2m}{\beta t^N}}.
\]
\item \textbf{Lower bound.}
For every $\varepsilon$ such that
\[
\frac{\delta(\mathcal D_{[L,m,M,C_{BV}]})}{\alpha}
\le \varepsilon \le
\frac{L}{8(1+2\alpha)}
\min\!\left\{
M,\;
\frac{m}{2L},\;
\frac{L}{8 t^N |f''(0)|}
\right\},
\]
one has
\begin{equation}\label{l_bound}
\mathcal{H}_{\varepsilon}\!\left(
I\big(S^N_{\Delta x}(\mathcal D_{\Delta x})\big)
\,\middle|\, L^1(\mathbb{R})
\right)
\ge
\frac{\Gamma^-}{\varepsilon},
\end{equation}
where
\[
\Gamma^- :=
\frac{L^2}{48 \ln(2)\, t^N |f''(0)|\,(1+2\alpha)}.
\]
\end{enumerate}
\end{theorem}
\begin{remark}
$\alpha \in (0,1]$ acts as a tolerance that manages a trade-off between the sharpness of the entropy lower bound and the admissible range of observation scales. Because the discrete scheme intrinsically introduces a forward error $\delta(\mathcal{D}_{[L,m,M,C_{BV}]})$ on the associated precursor class, it is impossible to perfectly resolve the metric complexity at scales finer than this numerical diffusion. The condition $\varepsilon \ge \delta(\mathcal{D}_{[L,m,M,C_{BV}]})/\alpha$ explicitly quantifies this limit. Consequently, $\alpha$ governs a structural tension: choosing a small $\alpha$ (strict tolerance) yields a sharper continuous-like lower bound constant $\Gamma^-$ (as $1+2\alpha \to 1$), but severely restricts the validity of the estimate. Conversely, relaxing the tolerance by choosing a larger $\alpha$ allows one to probe finer scales closer to the absolute numerical error floor, but comes at the cost of a slightly penalized entropy constant $\Gamma^-$.
\end{remark}

Theorem \ref{thm:main} provides a definitive, affirmative answer to \ref{Q1}. We demonstrate that numerical solutions can inherit the optimal $1/\varepsilon$ entropy scaling of the exact flow, provided the underlying method satisfies \ref{hyp:general_cond}--\ref{hyp:oslc}. Standard conservative monotone schemes, such as Lax-Friedrichs and Godunov methods, naturally fulfill these structural hypotheses \cite{Tadmor12,BO}. In Section~\ref{discussion}, we compare the constants of the continuous-time setting with those derived for the numerical schemes, discussing how the specific OSLC constants of different numerical methods explicitly modify the upper bound. Addressing \ref{Q2}, Theorem \ref{thm:main} exposes a fundamental, quantitative relationship between the information-theoretic resolution limit and the approximation error $\delta(\mathcal{D}_{[L,m,M,C_{BV}]})$. The theorem demonstrates that the complexity of the exact entropy solution set is successfully emulated by the discrete operator only at observation scales $\varepsilon$ strictly bounded from below. Regarding \ref{Q3}, in Section~\ref{discussion}, we distill the lower bound mechanism into a general transfer principle useful for other problems, and finally discuss the implications of this information-theoretic framework for the design and analysis of numerical methods.


\section{Proof of Theorem \ref{thm:main}}

For clarity of exposition, we divide the proof of Theorem~\ref{thm:main} into two independent propositions. The first (Proposition~\ref{prop:upper_estimate}) establishes the $\varepsilon$-entropy upper bound \eqref{u_bound} by exploiting the regularizing effect of the discrete OSLC. The second (Proposition~\ref{prop:lower_estimate}) derives the corresponding lower bound \eqref{l_bound} by demonstrating that the numerical operator uniformly approximates a strictly separated target family, thereby inheriting its metric complexity. Because these two mechanisms operate on fundamentally different principles, they impose distinct restrictions on the valid observation scales $\varepsilon$. The global two-sided estimate presented in Theorem~\ref{thm:main} is ultimately obtained by intersecting these respective validity ranges.

\subsection*{Upper estimate} 

For the proof we need two preliminary lemmas.

\begin{lemma}[$\varepsilon$-entropy of non-decreasing bounded functions; Lemma 3.1 \cite{DLG}]\label{lem:DLG}
For $L>0$ and $V>0$ set
\begin{equation}\label{eq:I_LV}
\mathcal I_{L,V}:=\{w:[0,L]\to[0,V]\ :\ w \ \text{is nondecreasing}\}.
\end{equation}
Then, for $0<\varepsilon\le \dfrac{L V}{6}$,
\begin{equation}\label{eq:DLG_entropy}
\mathcal H_{\varepsilon}\big(\mathcal I_{L,V}\ \big|\ L^1([0,L])\big)
\le \frac{4L V}{\varepsilon}.
\end{equation}
\end{lemma}

\begin{lemma}[Geometric $L^\infty$ bound for $BV$ functions; Lemma 4.2 \cite{AON1}]\label{lem:BV_Linfty}
Let $v\in BV(\mathbb R)$ be compactly supported and satisfy
\[
Dv \le B \quad\text{in the sense of measures}
\]
for some $B>0$. Then
\begin{equation}\label{eq:Linfty_from_BV}
\|v\|_{L^\infty}\le \sqrt{2 B \|v\|_{L^1}}.
\end{equation}
\end{lemma}
\begin{proposition}[$\varepsilon$-entropy upper bound for the numerical solutions]
\label{prop:upper_estimate}
For every $\Delta x>0$ and every scale $\varepsilon$ satisfying
\[
0<\varepsilon \le \frac{2L^N+\Delta x}{6}
\left(
\frac{2L^N+\Delta x}{\beta t^N}
+
\sqrt{\frac{2m}{\beta t^N}}
\right),
\]
the piecewise-linear reconstructed numerical solution set satisfies the $\varepsilon$-entropy upper bound
\[
\mathcal H_{\varepsilon}\!\left(
I\!\left(S_{\Delta x}^N(\mathcal D_{\Delta x})\right)\middle|L^1(\mathbb R)
\right)
\le \frac{\Gamma^+}{\varepsilon},
\]
where the constant $\Gamma^+$ is given by
\[
\Gamma^+ := \frac{4(2L^N + \Delta x)^2}{\beta t^N} + 4(2L^N + \Delta x)\sqrt{\frac{2m}{\beta t^N}}.
\]
\end{proposition}

\begin{proof}
By~\eqref{eq:fd} the discrete solution has compact support. In particular, $u^N_j=0$ for $|j|>K^N$, where $K^N=L^N/\Delta x$.
From the OSLC \ref{hyp:oslc} we have, for every index $j$,
\[
\frac{u^N_{j+1}-u^N_j}{\Delta x}\le \frac{1}{\beta t^N},
\]
hence the piecewise-linear interpolant $I(u^N)$ satisfies, on each cell,
\[
\partial_x I(u^N)(x)=\frac{u^N_{j+1}-u^N_j}{\Delta x}\le \frac{1}{\beta t^N}
\qquad \text{a.e. } x\in(x_j,x_{j+1}).
\]
Therefore $I(u^N)\in BV([-L^N,L^N])$ and, in the sense of measures,
\begin{equation}\label{der_above}
D I(u^N)\le \frac{1}{\beta t^N}.
\end{equation}

We now fix an index outside the support, for instance
\[
j_0:=-K^N-1,
\]
so that $u^N_{j_0}=0$ for every admissible datum $u_0\in \mathcal D_{\Delta x}$. Define the step function $R^N=\{R^N_j\}$ and the remainder $w^N=\{w^N_j\}$ on the indices $j=j_0,\dots,K^N$ by
\[
R^N_j:=\frac{(j-j_0)\Delta x}{\beta t^N},
\qquad
w^N_j:=R^N_j-u^N_j.
\]
With this choice $R^N_{j_0}=0$ and $w^N_{j_0}=0$. Moreover,
\[
w^N_{j+1}-w^N_j
=
\frac{\Delta x}{\beta t^N}-(u^N_{j+1}-u^N_j)\ge 0,
\]
by the OSLC, hence $w^N$ is non-decreasing and nonnegative on the index set.

By linearity of $I$, we can decompose the numerical solution as
\[
I(u^N)=I(R^N)-I(w^N).
\]
Since $I(R^N)$ is a fixed deterministic function, translation by $I(R^N)$ is an isometry in $L^1$, and therefore the $\varepsilon$-covering numbers of the family
\[
I\!\left(S^N_{\Delta x}(\mathcal D_{\Delta x})\right)=\{I(u^N)\}
\]
coincide with those of the remainder family $\{I(w^N)\}$. Therefore it suffices to estimate the entropy of $\{I(w^N)\}$ in $L^1([x_{j_0},x_{K^N}])$.

We now estimate the range of $I(w^N)$. Since $w^N$ is nondecreasing and $w^N_{j_0}=0$, we have
\[
0\le I(w^N)(x)\le w^N_{K^N}
\qquad \text{for a.e. } x\in[x_{j_0},x_{K^N}].
\]
Now
\[
w^N_{K^N}
=
R^N_{K^N}-u^N_{K^N}
=
\frac{(K^N-j_0)\Delta x}{\beta t^N}-u^N_{K^N}
=
\frac{2L^N+\Delta x}{\beta t^N}-u^N_{K^N},
\]
hence
\[
w^N_{K^N}
\le
\frac{2L^N+\Delta x}{\beta t^N}+\|u^N\|_{L^\infty}.
\]

By $L^1$-stability of the conservative monotone scheme (and compact support), one has
\[
\|u^N\|_{L^1}\le \|u^0\|_{L^1}\le m.
\]
Combining this with the distributional slope bound \eqref{der_above} and applying Lemma \ref{lem:BV_Linfty} yields
\[
\|u^N\|_{L^\infty}
\le
\sqrt{\frac{2\|u^N\|_{L^1}}{\beta t^N}}
\le
\sqrt{\frac{2m}{\beta t^N}}.
\]
Therefore
\[
0\le I(w^N)(x)\le V
\qquad\text{a.e. on }[x_{j_0},x_{K^N}],
\]
with
\[
V:=
\frac{2L^N+\Delta x}{\beta t^N}
+
\sqrt{\frac{2m}{\beta t^N}}.
\]
Since the length of the interval is $2L^N+\Delta x$, after a translation of the spatial variable we have
\[
\{I(w^N)\}\subset \mathcal I_{\,2L^N+\Delta x,\;V}.
\]
By Lemma \ref{lem:DLG} (for $0<\varepsilon\le \frac{(2L^N+\Delta x)V}{6}$), it follows that
\[
\mathcal H_{\varepsilon}\big(\{I(w^N)\}\mid L^1([x_{j_0},x_{K^N}])\big)
\le
\frac{4(2L^N+\Delta x)V}{\varepsilon}.
\]
Substituting the bound on $V$ yields
\[
\mathcal H_{\varepsilon}\big(I(S^N_{\Delta x}(\mathcal D_{\Delta x})) \mid L^1(\mathbb R)\big)
=
\mathcal H_{\varepsilon}\big(\{I(w^N)\}\mid L^1([x_{j_0},x_{K^N}])\big)
\le
\frac{4(2L^N+\Delta x)}{\varepsilon}
\left(
\frac{2L^N+\Delta x}{\beta t^N}
+
\sqrt{\frac{2m}{\beta t^N}}
\right).
\]
This gives the claim.
\end{proof}
\subsection*{Lower estimate}
As for the upper bound, in order to state and prove the proposition related to the lower bound, we need two lemmas (proofs in the Appendix).

\begin{lemma}[Continuous separated targets; adapted from Propositions 2.1 and 2.2 in \cite{AON1}]
\label{lem:continuous_separated}
Let $h$, $L_T$ as in \eqref{eq:h_param}, \eqref{eq:suppLT}, and let $\mathcal{A}_{[L_T, Lh, h, (2T|f''(0)|)^{-1}]}$ as in~\eqref{eq:Adef}. For every scale $0 < \varepsilon \le  Lh/8$, there exists a family $F_{\varepsilon} \subset \mathcal{A}_{[L_T, Lh, h, (2T|f''(0)|)^{-1}]} \subset S^T(\mathcal D_{[L,m,M]})$ of exact entropy solutions such that
\[
\|v-w\|_{L^1(\mathbb R)} > 2\varepsilon \qquad\text{for all distinct }v,w \in F_{\varepsilon},
\]
and its cardinality satisfies
\[
|F_{\varepsilon}| \ge 2^{\tilde \Gamma^- / \varepsilon},
\]
where
\[
\tilde \Gamma^- := \frac{L^2}{48 \ln(2)\, T |f''(0)|}.
\]
\end{lemma}

\begin{lemma}[Uniform bounded variation of precursors for the target space; adapted from Proposition 2.1 in \cite{AON1}]
\label{lem:uniform_tv_precursors}
For every target profile $v \in \mathcal{A}_{[L_T, Lh, h, (2T|f''(0)|)^{-1}]}$, 
any precursor $d_{v} \in \mathcal D_{[L,m,M]}$ satisfying 
\[
S^T(d_{v}) = v
\]
belongs to $BV(\mathbb R)$. Moreover, there exists a constant $C_{BV} > 0$, depending only on $L$, $T$, $h$, and the flux $f$, such that
\[
TV(d_{v}) \le C_{BV} \qquad\text{for all }v \in \mathcal{A}_{[L_T, Lh, h, (2T|f''(0)|)^{-1}]}.
\]
\end{lemma}

\begin{proposition}[$\varepsilon$-entropy lower bound for the numerical solutions]
\label{prop:lower_estimate}
Let $\alpha \in (0,1]$.
Then, for every $\Delta x>0$ and every scale $\varepsilon$ satisfying
\[
\frac{\delta(\mathcal D_{[L,m,M,C_{BV}]} )}{\alpha} \le \varepsilon \le \frac{L h}{8(1+2\alpha)},
\]
the piecewise-linear reconstructed numerical solution set satisfies the $\varepsilon$-entropy lower bound
\begin{equation}
\label{eq:lwb}
\mathcal H_\varepsilon\!\left(I\!\left(S_{\Delta x}^N(\mathcal D_{\Delta x})\right)\middle|L^1(\mathbb R)\right) \ge \frac{\Gamma^-}{\varepsilon},
\end{equation}
where $\Gamma^-:=\tilde \Gamma^-/(1+2\alpha)$, with $\tilde \Gamma^-$ given by Lemma \ref{lem:continuous_separated}.
\end{proposition}

\begin{proof}
By the assumptions on the scheme, the $L^1$-distance between the interpolated numerical solution and the exact entropy solution is bounded uniformly over sets of uniform bounded variation \cite{kuz,hr}. Since Lemma \ref{lem:uniform_tv_precursors} guarantees that all precursors $d_{v}$ of any $F_{\varepsilon}$ satisfy $TV(d_{v}) \le C_{BV}$, it holds that
\[
\max_{v\in F_{\varepsilon}}
\| I\!\left(S_{\Delta x}^N(P_{\Delta x}(d_v))\right) - v \|_{L^1(\mathbb R)} \le \delta(\mathcal D_{[L,m,M,C_{BV}]})
<\infty.
\]

Now fix $\alpha\in(0,1]$, and choose $\Delta x>0$ and $\varepsilon$ such that
\[
\frac{\delta(\mathcal D_{[L,m,M,C_{BV}]})}{\alpha} \le \varepsilon \le \frac{L_Th}{6(1+2\alpha)},
\]
Set a corresponding target scale as $\tilde\varepsilon := (1+2\alpha)\varepsilon$. Since $\tilde\varepsilon\in(0,L_Th/6]$, the family $F_{\tilde\varepsilon}$ provided by Lemma \ref{lem:continuous_separated} is well-defined and strictly $2\tilde\varepsilon$-separated in $L^1(\mathbb R)$.

For each target $v\in F_{\tilde\varepsilon}$, letting $d_v$ be a corresponding precursor provided by Lemma~\ref{lem:uniform_tv_precursors}
we define its corresponding discrete realization as
\[
v_{\Delta x} := I\!\left(S_{\Delta x}^N(P_{\Delta x}(d_{v}))\right)\in I\!\left(S_{\Delta x}^N(\mathcal D_{\Delta x})\right).
\]
By the uniform error estimate and the definition of the range for $\varepsilon$, we have
\[
\| v_{\Delta x} - v \|_{L^1(\mathbb R)} \le \delta(\mathcal D_{[L,m,M,C_{BV}]}) \le \alpha\varepsilon
\qquad\text{for all }v\in F_{\tilde\varepsilon}.
\]

Now, let $v,w\in F_{\tilde\varepsilon}$ be distinct targets ($v\neq w$). By the reverse triangle inequality, the distance between their discrete realizations satisfies
\begin{align*}
\| v_{\Delta x} - w_{\Delta x} \|_{L^1(\mathbb R)}
&\ge \| v - w \|_{L^1(\mathbb R)} - \| v_{\Delta x} - v \|_{L^1(\mathbb R)} - \| w_{\Delta x} - w \|_{L^1(\mathbb R)} \\
&> 2\tilde\varepsilon - 2\alpha\varepsilon \\
&= 2(1+2\alpha)\varepsilon - 2\alpha\varepsilon \\
&= 2(1+\alpha)\varepsilon \\
&> 2\varepsilon.
\end{align*}

Therefore, the finite numerical family
\(
\{\,v_{\Delta x}: v\in F_{\tilde\varepsilon}\,\}
\subset I\!\left(S_{\Delta x}^N(\mathcal D_{\Delta x})\right)
\)
is strictly $2\varepsilon$-separated in $L^1(\mathbb R)$. Hence, any $\varepsilon$-cover of the discrete image $I(S_{\Delta x}^N(\mathcal D_{\Delta x}))$ must contain at least $|F_{\tilde\varepsilon}|$ balls. This implies
\[
N_\varepsilon\!\left(I\!\left(S_{\Delta x}^N(\mathcal D_{\Delta x})\right)\middle| L^1(\mathbb R)\right)
\ge |F_{\tilde\varepsilon}| \ge 2^{\frac{\tilde \Gamma^-}{(1+2\alpha)\varepsilon}}.
\]
Hence, the claim.
\end{proof}
\section{Discussion and Open Problems}\label{discussion}
In this final section, we step back from the technical proofs to contextualize our results. We first compare our estimates with the established bounds for the exact entropy solution set to highlight the parallels. Subsequently, we clarify the practical implications of these entropy bounds, specifically regarding Lax's definition of high-resolution schemes and a posteriori error recovery, and conclude by outlining open questions.
\subsection{Comparison with Continuous Entropy Estimates}

For the exact problem, De Lellis and Golse \cite{DLG} provided an upper bound on the $\varepsilon$-entropy of $S^T(\mathcal{D}_{[L,m,M]})$, while Ancona, Glass, and Nguyen \cite{AON1} established a matching lower bound. To highlight the structural parallels between the exact flow and the numerical scheme, we present their respective bounds side-by-side:

\begin{align}
    & \text{\textbf{Continuous Exact Flow}} && \text{\textbf{Numerical Scheme}} \nonumber\\[2ex]
    \text{\textbf{Upper Bound:}} \quad
    & \frac{4}{\varepsilon} \left( \frac{4 L(T)^2}{cT} + 4 L(T) \sqrt{\frac{2m}{cT}} \right) && 
     \frac{1}{\varepsilon} \left( \frac{4(2L^N + \Delta x)^2}{\beta t^N} + 4(2L^N + \Delta x)\sqrt{\frac{2m}{\beta t^N}}\right) \label{eq:upper_comp} \\[3ex]
    \text{\textbf{Lower Bound:}} \quad
    & \frac{1}{\varepsilon} \left( \frac{L^2}{48 \ln(2)T |f''(0)|} \right) && 
    \frac{1}{\varepsilon} \left( \frac{L^2}{48 \ln(2)T |f''(0)|(1+2\alpha)} \right) \label{eq:lower_comp}
\end{align}

By comparing the bounds (with $L^N:=L+N \Delta x$, $t^N:=N \Delta t = T$
in the numerical scheme), three structural features stand out:

\begin{enumerate}
    \item \textit{The regularizing mechanism ($c$ vs.\ $\beta$):} In the upper bound for the exact solution set \eqref{eq:upper_comp}, the entropy depends inversely on the convexity parameter $c \le f''(u)$, which governs the Oleinik-type one-sided Lipschitz regularization \cite{Oleinik}. In the bound for the numerical solution set, the same role is played by the discrete OSLC constant $\beta$, through $\ell^+(u^n) \le 1/(\beta t^n)$. Thus, the upper metric complexity of the numerical set is controlled by the scheme's discrete regularization strength. This is the correct analogue of the continuous mechanism, although the explicit prefactor is scheme-dependent.

\item \textit{Support propagation ($L(T)$ vs.\ the discrete support scale):} The upper bound for the continuous counterpart depends on the exact support length $L(T)$ of the exact solution at time $T$. In the discrete setting, the proof uses an effective support scale of order $2L^N$, with only a lower-order grid correction of size $\Delta x$. Thus, the discrete estimate has the same quadratic dependence on the support length as the continuous one, and the additional $\Delta x$ term is a harmless geometric artefact rather than a change of principle. 

    \item \textit{Lower bound inheritance up to a tolerance factor:} The lower bound for the numerical solution set retains the same $1/\varepsilon$ scaling and the same dependence on $|f''(0)|$ as the lower bound of \cite{AON1} , but with the additional factor $(1+2\alpha)^{-1}$ coming from the comparison between exact and numerical solutions. In this sense, the discrete scheme inherits the entropy complexity of the exact solutions at scales above the resolution threshold $\varepsilon \ge \delta(\mathcal D_{[L,m,M,C_{BV}]})/\alpha$.
\end{enumerate}

\subsection{Remarks on the Entropy estimates}

It is important to emphasize the conceptual meaning and the limitations of entropy-based diagnostics. The quantity $\mathcal H_\varepsilon$ measures the number of distinguishable patterns retained by the numerical solution set at scale $\varepsilon$ (covering numbers), but it is insensitive to systematic, structure-preserving biases that shift or reparametrize those patterns. For example, if a scheme yields accurate approximations, a scheme obtained by translating its solutions by a nonzero value preserves the same $\varepsilon$-entropy as the original, whenever the translation is an isometry in the metric under consideration (e.g.\ in $L^1(\mathbb R)$). Nevertheless, the translated scheme may be arbitrarily inaccurate in norm, because $\mathcal H_\varepsilon$ does not control pointwise or normwise alignment.

This observation clarifies the meaning of ``high resolution'': it is a necessary condition for the \emph{possibility} of a posteriori recovery of accuracy because no information at scale $\varepsilon$ was lost. By contrast, if the numerical map collapses many exact solutions into a small number of numerically indistinguishable representatives at scale $\varepsilon$ (e.g.\ through excessive numerical diffusion that renders different shock positions identical up to $\varepsilon$), then the information has been irreversibly lost.

A complementary perspective is provided by the a posteriori arguments of Bressan, Chiri, and Shen \cite{BCS}. Their analysis shows that, once an approximate solution is available, one may recover an error bound by post-processing the output: first verifying a uniform total variation bound, then locating the relevant large shocks, and finally checking that the remaining regions have small oscillation. This is not an entropy argument, but it is highly relevant here because it shows that accurate recovery can be achieved from a numerically computed profile that is not itself pointwise accurate, provided that the correct structural information is still present in the output.

In that sense, the present entropy criterion and the post-processing approach are complementary. The former detects whether the discrete family has retained the correct amount of information at scale $\varepsilon$; the latter explains how such information can then be extracted to obtain a posteriori accuracy. 

Finally, the lower-bound mechanism proved in Proposition~\ref{prop:lower_estimate} is even more flexible. It does not depend on monotonicity or on the OSLC per se, but on a uniform approximation estimate on a precursor class with bounded variation. Consequently, whenever a PDE (not necessarily a conservation law) admits a family of pairwise $2\varepsilon$-separated exact states used to prove a related entropy lower bound, and a numerical method approximates that family uniformly, the numerical solution set inherits the same lower-bound scaling up to the approximation tolerance. In particular, for any scheme, including suitable higher-order schemes, for which such a uniform estimate is available, the continuous and discrete problem complexities match at the level of scaling. We can therefore state the following theorem:

\begin{theorem}[$\varepsilon$-Entropy Transfer Principle]
\label{thm:entropy_transfer}
Let $(X,d)$ be a metric space and let $S^T: \mathcal{D} \to X$ be the exact flow of an evolution equation on a set of initial data $\mathcal{D}$. Let $P_{\Delta x}$, $S^N_{\Delta x}$, and $I$ denote the projection, discrete evolution, and interpolation operators respectively, defining the numerical solution set $I\big(S^N_{\Delta x}(\mathcal{D}_{\Delta x})\big)$, where $\mathcal{D}_{\Delta x} = P_{\Delta x}(\mathcal{D})$. Suppose there exists a subset $\mathcal{D}^* \subset \mathcal{D}$ satisfying:
\begin{enumerate}
    \item \emph{Continuous separated targets:} There exist a threshold $\varepsilon_0 > 0$ and a function $C: (0, \varepsilon_0] \to [1, \infty)$ such that, for every $\varepsilon \in (0, \varepsilon_0]$, there is a family of exact solutions $F_\varepsilon \subset S^T(\mathcal{D}^*)$, with cardinality bounded from below by $|F_\varepsilon| \ge C(\varepsilon)$, satisfying
    \[
    d(v,w) > 2\varepsilon \qquad \text{for all distinct } v, w \in F_\varepsilon.
    \]
    
    \item \emph{Uniform approximation:} The numerical method uniformly approximates the exact solution over $\mathcal{D}^*$ with a finite supremum of the approximation error $\delta(\mathcal D^*) \to 0$ when $\Delta x \to 0$:
    \[
    \sup_{d \in \mathcal{D}^*} d\Big(S^T(d), I\big(S^N_{\Delta x}(P_{\Delta x}(d))\big)\Big) \le \delta(\mathcal D^*),
    \qquad T=N\Delta t.
    \]
\end{enumerate}

Then, for any $\alpha \in (0,1]$ and $\varepsilon$ such that
\[
\frac{\delta(\mathcal D^*)}{\alpha} \le \varepsilon \le \frac{\varepsilon_0}{1+2\alpha},
\]
we have
\[
\mathcal{H}_\varepsilon\Big(I\big(S^N_{\Delta x}(\mathcal{D}_{\Delta x})\big) \mid X\Big) \ge \log_2 C\big((1+2\alpha)\varepsilon\big).
\]
\end{theorem}
\begin{proof}
The argument is analogous to that of Proposition~\ref{prop:lower_estimate}. Set \(\tilde{\varepsilon}:=(1+2\alpha)\varepsilon\). By uniform approximation, each \(v\in F_{\tilde{\varepsilon}}\) has a numerical counterpart within distance at most \(\alpha\varepsilon\), so the reconstructed family is \(2\varepsilon\)-separated in \(X\). Hence \(N_\varepsilon\big(I(S^N_{\Delta x}(\mathcal D_{\Delta x}))\mid X\big)\ge |F_{\tilde{\varepsilon}}|\ge C(\tilde{\varepsilon})\).
\end{proof}
\subsection{Future directions}
Several concrete directions for future research emerge. While the lower bound mechanism may be extended to other schemes (provided a uniform approximation estimate on a suitable precursor class is available), establishing a tight upper bound for second- or higher-order methods remains an open challenge. The upper estimate derived here inherently relies on the discrete OSLC, which is a hallmark of first-order monotone methods. Since higher-order schemes circumvent the strict OSLC to achieve better accuracy, new tools are required to bound the metric complexity of their numerical solution sets.

The present analysis strongly depends on the uniform convexity of the flux function, which dictates the regularizing properties of both the PDE and its discrete approximations. Extending our entropy estimates to scalar conservation laws with general, non-convex fluxes would provide a numerical counterpart to the compactness results established by Ancona, Glass, and Nguyen \cite{AON4}. A natural next question is also whether the same information-theoretic viewpoint can be extended to scalar balance laws  \cite{AON1}.

The information-theoretic perspective should also be extended to other classes of nonlinear PDEs, such as Hamilton--Jacobi equations. The $\varepsilon$-entropy for Hamilton--Jacobi equations has already been investigated (see, e.g., \cite{AON3}), but the numerical counterpart remains largely unexplored. In this
setting, as for conservation laws, one may ask which convergent schemes preserve
the correct entropy scaling over a meaningful range of observation scales, and
which schemes lose part of the metric complexity of the exact solution family.

From a computational standpoint, the theoretical guarantee that a scheme is ``high-resolution'' in Lax's sense \cite{Lax78} strongly motivates further algorithmic development. In particular, it suggests studying how high-resolution discrete outputs may be combined with suitable a posteriori post-processing techniques to exploit the preserved structural complexity and improve the quality of the reconstructed solution; see, for instance, \cite{BCS,NT,Tadmor89} for representative examples of post-processing in related contexts. 
 
Finally, a more speculative but potentially fruitful direction concerns data-driven models. For instance, one may ask whether a neural network (NN) trained on a structured family of initial data, for example piecewise constant functions, can still approximate the exact solutions on the same family and, under suitable additional generalization assumptions, for more general admissible data. If the training family has the same $\varepsilon$-entropy scaling as the full admissible class, and if the NN satisfies a suitable uniform approximation estimate on that family, then one may expect the corresponding output set to retain the same information-theoretic resolution, at least up to the approximation tolerance. In this sense, even without explicitly solving the equation, a suitably trained NN could still provide approximations that remain close to the exact solution at the corresponding observation scale. Future research in the systematic study of metric entropy in numerical analysis may thus serve not only as a diagnostic tool, but also as a criterion for designing novel numerical schemes.

\appendix
\section{Proofs of Auxiliary Results}
\addcontentsline{toc}{section}{Appendix Auxiliary proofs and technical lemmas}
\begin{proof} [Proof of Theorem \ref{lax_thm}]

By the definition of $\varepsilon$-entropy, there exist $\tilde N= 2^{\mathcal{H}_{\delta(\mathcal D)}(I(S_{\Delta x}^{N}(\mathcal D_{\Delta x})) \mid X)}$ balls in $X$ of radius $\delta(\mathcal D)$, centered at $\{u_1, u_2, \dots, u_{\tilde N}\}$, which cover the set of numerical reconstructions $(I(S_{\Delta x}^{N}(\mathcal D_{\Delta x}))$. Thus, for any $v \in(I(S_{\Delta x}^{N}(\mathcal D_{\Delta x})) $, there exists an index $j \in \{1, \dots, \tilde N\}$ such that:
\begin{equation}\label{eq:proof1}
\| v - u_j \| \le \delta(\mathcal D).
\end{equation}
Now, consider an arbitrary exact solution $s \in S^T(\mathcal D)$. By definition, there exists a datum $d \in \mathcal{D}$ such that $s = S^T(d)$. Let $v = I(S^N_{\Delta x}(u_0))$, $u_0=P_{\Delta x}(d)$, be the corresponding numerical reconstruction generated by the projection of the same initial datum. According to our definition of the error in \eqref{eq:delta_def}, we have:
\begin{equation}\label{eq:proof2}
\| s - v \| \le \delta(\mathcal D).
\end{equation}
By the triangle inequality, we can estimate the distance between the exact solution $s$ and the center $u_j$ that covers its numerical counterpart $v$:
\[
\| s - u_j \| \le \| s - v \| + \| v - u_j \|.
\]
Substituting the inequalities from \eqref{eq:proof1} and \eqref{eq:proof2}, we obtain:
$\| s - u_j \| \le \delta(\mathcal D) + \delta(\mathcal D) = 2\delta(\mathcal D).$ This shows that every element of the exact solution set $\mathcal{S}^T$ is contained within a ball of radius $2\delta(\mathcal D)$ centered at one of the $\tilde N$ points $\{u_j\}$. Consequently, the collection of $\tilde N$ balls of radius $2\delta(\mathcal D)$ forms a valid covering for $\mathcal{S}^T$. It follows that the minimal number of balls of radius $2\delta(\mathcal D)$ required to cover $S^T(\mathcal D)$, denoted by $N_{2\delta(\mathcal D)}(S^T(\mathcal D))$, satisfies:
\[\mathcal{H}_{2\delta(\mathcal D)}(S^T(\mathcal D)) \mid X)=\log_2N_{2\delta(\mathcal D)}(S^T(\mathcal D)) \le \log_2 \tilde N =\mathcal{H}_{\delta(\mathcal D)}(\mathcal{I}(\mathcal{S}_{\Delta x}^{N}(\mathcal D_{\Delta x}))  \mid X). \]
\end{proof}
\begin{proof}[Proof of Lemma \ref{lem:continuous_separated}]
The existence of the strictly $2\varepsilon$-separated target family follows directly from the combinatorial construction detailed in the proof of Proposition 2.2 in \cite{AON1}. 

Given any $n\in\N$, $h\geq \varepsilon/L_T$, the authors define a finite set $\mathcal{F} \subset \mathcal{A}_{[L_T, Lh, h, (2T|f''(0)|)^{-1}]}$ containing~$2^n$ piecewise-affine profiles parameterized by binary sequences $\iota \in \{-1,1\}^n$. By relating the $L^1$-distance between these profiles to the Hamming distance of their generating sequences, they prove that the maximum number of profiles in $\mathcal{F}$ lying within an $L^1$-distance of $2\varepsilon$ from any given function in $\mathcal{F}$ is bounded by a combinatorial sum $\mathcal{C}(2\varepsilon)$.

To explicitly construct the strictly $2\varepsilon$-separated family $F_{\varepsilon}$, we apply a standard greedy extraction algorithm: we iteratively select a profile from $\mathcal{F}$ to include in $F_{\varepsilon}$ and discard all its neighbors within a closed $2\varepsilon$-ball. Since at most $\mathcal{C}(2\varepsilon)$ elements are discarded at each step, the extraction is guaranteed to produce a separated subset whose cardinality is at least
\[
|F_{\varepsilon}| \ge \frac{2^n}{\mathcal{C}(2\varepsilon)}.
\]

By optimizing the parameters $h$, $n$, for the given $\varepsilon$, and using $L_T \ge \frac{3}{4} L$ (which is intrinsically guaranteed by the upper bound~\eqref{eq:h_param} on the fixed parameter $h$), the ratio bounding the cardinality is shown to satisfy
\[
\frac{2^{ \overline n}}{\mathcal{C}(2\varepsilon)} \ge 2^{\tilde \Gamma^- / \varepsilon},
\] for some $\overline n$.

Finally, the existence of exact entropy precursors $d_{v} \in \mathcal D_{[L,m,M]}$ ensuring that $F_{\varepsilon} \subset S^T(\mathcal D_{[L,m,M]})$ is obtained via the density argument of Proposition 2.1 in \cite{AON1}, which guarantees the full inclusion $\mathcal{A}_{[L_T, Lh, h, (2T|f''(0)|)^{-1}]} \subset S^T(\mathcal D_{[L,m,M]})$.
\end{proof}
\begin{proof}[Proof of Lemma \ref{lem:uniform_tv_precursors}]
We argue by density. Fix $v\in \mathcal{A}_{[L_T, Lh, h, (2T|f''(0)|)^{-1}]}$.
By the proof of the Proposition~2.1 of \cite{AON1}, there exists a sequence
\[
v^n \in \mathcal{A}_{[L_T, 2Lh, h, (2T|f''(0)|)^{-1}]}\cap C^1(\mathbb R)
\]
such that $v^n \to v$ in $L^1(\mathbb R)$, and for each $n$ one can find
$d^n \in \mathcal{D}_{[L,m,M]}$ satisfying $S^T(d^n)=v^n$.

The enlargement from $Lh$ to $2Lh$ is required by the approximation procedure. In the constructive proof, one writes
\[
d^n(x)=w^n(T,-x),
\]
where $w^n$ is the forward entropy solution starting from the reflected terminal datum $v^n(-\cdot)$. The  argument yields a uniform bound on the spatial derivative at time $T$, namely
\begin{equation}\label{uniform_der}
\|w_x^n(T,\cdot)\|_{L^\infty(\mathbb R)}\le C_*,
\end{equation}
where $C_*$ is only dependent of the parameters defining the class $\mathcal{A}_{[L_T, 2Lh, h, (2T|f''(0)|)^{-1}]}$. Since each $d^n$ is supported in $[-L,L]$, it follows that
\[
TV(d^n)
\le \int_{-L}^{L} \bigl| (d^n)'(x)\bigr|\,dx
= \int_{-L}^{L} \bigl| w_x^n(T,-x)\bigr|\,dx
\le 2L\,\|w_x^n(T,\cdot)\|_{L^\infty(\mathbb R)}
\le 2L\,C_*.
\]
Thus, setting
\[
C_{BV}:=2L\,C_*,
\]
we obtain
\[
TV(d^n)\le C_{BV}
\qquad\text{for all }n\in\mathbb N,
\]
with $C_{BV}$ independent of $n$ and $v$, since it is the same for every element in $\mathcal{A}_{[L_T, 2Lh, h, (2T|f''(0)|)^{-1}]}$.

Since $(d^n)$ is uniformly bounded in $BV(\mathbb R)$ and shares a common compact support in $[-L,L]$, the compact embedding $BV([-L,L])\hookrightarrow L^1([-L,L])$ yields, up to a non-relabeled subsequence, a function
\[
d_{v}\in BV(\mathbb R)
\]
such that
\[
d^n\to d_{v}\quad\text{in }L^1(\mathbb R).
\]
By lower semicontinuity of the total variation,
\[
TV(d_{v})\le \liminf_{n\to\infty}TV(d^n)\le C_{BV}.
\]

Moreover, since each $d^n$ belongs to $\mathcal D_{[L,m,M]}$, the limit $d_{v}$ inherits the support, $L^1$, and $L^\infty$ bounds; hence,
\[
d_{v}\in \mathcal D_{[L,m,M]}.
\]

Finally, the $L^1$-continuity of the entropy semigroup gives
\[
S^T(d_{v})
=
\lim_{n\to\infty} S^T(d^n)
=
\lim_{n\to\infty} v^n
=
v.
\]
\end{proof}

\clearpage

\end{document}